\documentclass[11pt]{amsart}
\usepackage{amsmath}
\usepackage{amsthm}
\usepackage{mathrsfs}
\usepackage{latexsym}
\usepackage{amssymb}
\usepackage{amsthm}
\usepackage{amscd}
\usepackage[dvips]{graphics}
\usepackage[all,cmtip]{xy}
\usepackage{cancel}
\usepackage{xcolor}
\usepackage[hidelinks]{hyperref}

\usepackage{fourier}

\DeclareSymbolFont{AMSb}{U}{msb}{m}{n}
\DeclareMathSymbol{\N}{\mathbin}{AMSb}{"4E}
\DeclareMathSymbol{\Z}{\mathbin}{AMSb}{"5A}
\DeclareMathSymbol{\R}{\mathbin}{AMSb}{"52}
\DeclareMathSymbol{\Q}{\mathbin}{AMSb}{"51}
\DeclareMathSymbol{\I}{\mathbin}{AMSb}{"49}
\DeclareMathSymbol{\C}{\mathbin}{AMSb}{"43}

\DeclareFontFamily{U}{mathx}{\hyphenchar\font45}
\DeclareFontShape{U}{mathx}{m}{n}{
      <5> <6> <7> <8> <9> <10>
      <10.95> <12> <14.4> <17.28> <20.74> <24.88>
      mathx10
      }{}
\DeclareSymbolFont{mathx}{U}{mathx}{m}{n}
\DeclareFontSubstitution{U}{mathx}{m}{n}
\DeclareMathAccent{\widecheck}{0}{mathx}{"71}
\DeclareMathAccent{\wideparen}{0}{mathx}{"75}

\newcommand{\dbl}{[\hspace{-0.2ex}[}
\newcommand{\dbr}{]\hspace{-0.2ex}]}
\newcommand{\db}[1]{\dbl {#1} \dbr}

\newcommand{\iso}{\cong}
\newcommand{\invlim}{\underleftarrow{\textnormal{lim}}\,}

\newcommand{\Hom}{\textnormal{Hom}}

\newcommand{\GL}{\textnormal{GL}}

\newcommand{\tn}[1]{\textnormal{#1}}

\newcommand{\F}{\mathbb{F}}
\newcommand{\bigO}{\mathcal{O}}

\newcommand{\Tr}{\operatorname{Tr}}
\newcommand{\oTr}{\operatorname{oTr}}
\newcommand{\rank}{\operatorname{rank}}
\newcommand{\Stab}{\operatorname{Stab}}

\newcommand{\Ker}{\operatorname{Ker}}

\newcommand{\Ext}{\operatorname{Ext}}
\newcommand{\PC}{\operatorname{PC}}
\newcommand{\Res}{\operatorname{Res}}

\newcommand{\Aut}{\operatorname{Aut}}

\numberwithin{equation}{section}

\title{Finiteness properties of profinite blocks and blocks with defect group $\Z_p^n$}
\author{Florian Eisele}
\address[Florian Eisele]{Department of Mathematics, University of Manchester, Oxford Road, Manchester  M13 9PL, United Kingdom}
\email{florian.eisele@manchester.ac.uk}
\author{John W. MacQuarrie}
\address[John W. MacQuarrie]{Department of Mathematics, Av. Antônio Carlos, 6627 - Pampulha. Belo Horizonte - MG31270-901, Brazil}
\email{john@mat.ufmg.br}

\begin{document}

\newtheorem{defn}[equation]{Def{i}nition}
\newtheorem{prop}[equation]{Proposition}
\newtheorem{lemma}[equation]{Lemma}
\newtheorem{theorem}[equation]{Theorem}
\newtheorem{corol}[equation]{Corollary}
\newtheorem{question}[equation]{Question}
\newtheorem{notation}[equation]{Notation \& Conventions}

\theoremstyle{definition}
\newtheorem{example}[equation]{Example}

\theoremstyle{definition}
\newtheorem{examples}[equation]{Examples}

\theoremstyle{definition}
\newtheorem{remark}[equation]{Remark}

\begin{abstract}
    We prove that any block of a profinite group with a topologically finitely generated defect group $D$ is isomorphic to a block of a virtually pro-$p$ group, and if $D$ is finitely presented then so is the basic algebra of the block over $k$. Using these general results we then prove that any block with defect group $D=\Z_p^n$ is Morita equivalent to $\bigO_\alpha \db{\Z_p^n\rtimes E}$ for a finite $p'$-group $E$ acting faithfully and a $2$-cocycle $\alpha\in H^2(E,k^\times)$. In particular, Donovan's conjecture holds for such blocks of profinite groups.
\end{abstract}

\maketitle

\section{Introduction}

Let $\bigO$ be a complete discrete valuation ring of characteristic $0$ with residue field $k=\bar k$ of characteristic $p>0$. We are interested in blocks of profinite groups, both over $k$ and over $\bigO$. In general, such blocks are rather complicated pseudocompact algebras. However, just as for blocks of finite groups, we expect \emph{defect groups} to have some control over how complicated the algebra structure of a block is. Blocks of cyclic defect are a particularly striking instance of this phenomenon. Blocks of finite (and profinite) groups with cyclic defect groups are known to be Brauer tree algebras, which are well-understood. Surprisingly, blocks whose defect groups are infinite pro-cyclic, that is, blocks with defect groups isomorphic to $\Z_p$, turned out to have an even simpler structure: they are infinite-dimensional Brauer star algebras \cite{JohnRicardoCyclic}. The motivation for the present paper was to understand \emph{why} blocks with defect group $\Z_p$ have such a simple structure and how this generalises to other pro-$p$ defect groups. There are two concrete questions we set out to answer. Firstly, why are blocks with defect group $\Z_p$ Morita equivalent to a completed path algebra of a \emph{finite} quiver modulo a \emph{finite} number of relations? And secondly, why are their Brauer trees \emph{stars}, mirroring the situation in blocks of $p$-solvable finite groups with cyclic defect groups? 

Our answer to the first question is quite general. It shows that if $D$ is finitely generated, then it suffices to study blocks of virtually pro-$p$ groups instead of general profinite groups, and finite presentability of $D$ translates directly into finite presentability of the corresponding basic algebras of blocks. These results should prove useful in other contexts as well.
\begin{theorem}
    Let $\bigO \db{G}b$ be a block of a profinite group $G$ with defect group $D$.
    \begin{enumerate}
        \item The block $\bigO \db{G}b$ is isomorphic to a block $\bigO \db{H} c$ of a \emph{virtually pro-$p$ group} $H$ if and only if $\bigO \db{G}b$ has finitely many simple modules.
        \item If $D$ is finitely generated as a pro-$p$ group, then $\bigO \db{G}b$ has finitely many simple modules, and a basic algebra of $k\db{G}b$ is isomorphic to $k\db{Q}/I$ for a finite quiver $Q$ and some closed ideal $I\subseteq J^2(k\db{Q})$, where $J(k\db{Q})$ is the Jacobson radical of $k\db{Q}$.
        \item If $D$ is finitely presented as a pro-$p$ group, then a basic algebra of $k\db{G}b$ is isomorphic to $k\db{Q}/I$ for a finite quiver $Q$ and a finitely generated ideal $I\subseteq J^2(k\db{Q})$.
    \end{enumerate}
\end{theorem}

Our answer to the second question relies on reduction theorems for Donovan's conjecture from the modular representation theory of finite groups. We show that these reduction results generalise to profinite groups without any restrictions on the defect group. For blocks with defect groups isomorphic to $\Z_p^n$, the reductions show that any such block is Morita equivalent to a block with normal defect group. The reductions work much more cleanly than analogous reductions in the finite case (the classification of finite simple groups is not needed at all), among other things because a finite normal subgroup cannot intersect the defect group $\Z_p^n$ non-trivially. See Lemma~\ref{lemma:group-theory-caseAB} for the group-theoretic fact we use. Our result is the following, which completely classifies the blocks in question:

\begin{theorem}[see Theorem~\ref{thm:reduction}]
    Let $\bigO \db{G}b$ be a block of a profinite group  with defect group isomorphic to $\Z_p^n$ for some $n\in \N$. Then 
    $$
        \bigO\db{G}b \sim_{\rm Morita} \bigO_\alpha\db{\Z_p^n\rtimes E}
    $$
    for a finite $p'$-group $E$ acting faithfully on $\Z_p^n$ and some $2$-cocycle $\alpha\in H^2(E,k^\times)$.
\end{theorem}

For pro-cyclic defect $\Z_p$ this recovers the result of \cite{JohnRicardoCyclic} and generalises it from $k$ to $\bigO$, since in that case $E\leq C_{p-1}$ and therefore $H^2(E, k^\times)=0$ (see \cite[Proposition~1.2.10]{LinckelmannBookVol1}). The result above furthermore shows that a version of Donovan's conjecture holds for blocks of defect $\Z_p^n$. In the finite group case, Donovan's conjecture predicts that there are only finitely many blocks with any given defect group, up to Morita equivalence. The profinite analogue is obviously false in general. For instance, taking $\bigO \db{(\prod_\N C_p)\rtimes E}$ for various $p'$-groups $E$ shows that there are infinitely many non-equivalent blocks with defect group $D=\prod_\N C_p$. But it does hold for $D=\Z_p^n$ by the theorem above, which raises the following question:
\begin{question}[Donovan problem for profinite groups]\label{question:donovan}
    For which pro-$p$ groups $D$ are there only finitely many blocks with defect group $D$, up to Morita equivalence?
\end{question}
For this question to have a positive answer for any given $D$, it is necessary that $\Aut(D)$ has only finitely many conjugacy classes of $p'$-subgroups (that is, there must be at most finitely many possibilities for the ``inertial quotient''). However, this may not be sufficiently restrictive since the inertial quotient is just one invariant determined by the block pro-fusion system. In light of this, it may be sensible to alter the question slightly: for a fixed $D$ and a fixed saturated pro-fusion system $\mathcal F$ on $D$, are there at most finitely many blocks with block pro-fusion system $\mathcal F$, up to Morita equivalence?
By \cite{ProfusionPaper} this holds if $\mathcal F$ is trivial, and by the present paper this holds if $D\cong \Z_p^n$ and $\mathcal F$ is arbitrary. In full generality, Question~\ref{question:donovan} is obviously hard, since it contains Donovan's conjecture as a special case. But the reductions of \S\ref{section:moritareductions} work surprisingly well for $D=\Z_p^n$, mainly because this group is torsion-free, and the same or similar methods may well apply to other classes of torsion-free defect groups.

\subsection*{Notation \& Conventions}
    \begin{enumerate}
    \item Throughout, $p>0$ is a prime and $k$ is an algebraically closed field of characteristic $p$ endowed with the discrete topology. We let $\bigO$ denote an unramified complete discrete valuation ring of characteristic $0$ such that $\bigO/J(\bigO)\cong k$, endowed with the $J(\bigO)$-adic topology.
    \item Given a profinite group $G$, the completed group algebras $\bigO\db{G}$ and $k\db{G}$ are \emph{pseudocompact algebras}, and we will freely use their properties.  For general background on pseudocompact algebras, see for instance \cite{Brumer, JohnPavelPeterInfGen, IusenkoMacQuarrieSemisimple}. We use the word ``module'' to mean ``pseudocompact module'' unless stated otherwise.
    \item For a pseudocompact algebra $A$, $J(A)$ denotes the Jacobson radical of $A$ -- by definition this is the intersection of the maximal closed left ideals of $A$, but $J(A)$ coincides with the Jacobson radical of $A$ as an abstract algebra (\cite[Proposition 3.2]{IusenkoMacQuarrieSemisimple}).  On the other hand it is not the case that the abstract product $J(A)\cdot J(A)$ is closed (\cite[Example 2.14]{IusenkoMacQuarrieSemisimple}).  {Given closed ideals $I, L$ of $A$, we denote by $I^2$ and $IL$ the topological closures of the abstract products $I\cdot I$ and $I\cdot L$, respectively}.  A closed ideal $I$ of a pseudocompact algebra $A$ will be called a \emph{relation ideal} if $I\subseteq J^2(A)$.
    \item Any ``$G$-algebras'' we consider are $G$-algebras over $\bigO$. We consider $G$-algebras over $k$ to be a special case of this. 
    \item Given a profinite group $G$ and closed normal subgroups $N\subseteq M$ of $G$ we let
    $$
        \varphi_N:\ \bigO\db{G} \longrightarrow \bigO\db{G/N}\quad\textrm{and}\quad \varphi_{MN}:\ \bigO\db{G/N} \longrightarrow \bigO\db{G/M}
    $$
    denote the canonical surjections. 
    \item For a group $G$ acting on an abelian group $A$ and subgroup $H\leq G$ of finite index we define $\Tr_H^G(x)=\sum_{gH\in G/H} {}^gx$ for any $x\in A^H$ (the set of $H$-invariants of $A$).
    \end{enumerate}

\section{Preliminaries}

In this section we will give some basic facts about blocks and defect groups. Defect groups of blocks of profinite groups were first defined in \cite{JohnRicardoBrauer1}, and that paper is a good reference for most of their standard properties. The coefficient ring in \cite{JohnRicardoBrauer1} is assumed to be a field, but the results we require hold for $\bigO$, in light of Proposition \ref{prop:defectgroupsoveroarethesame} below.  We will often use, without further reference, that defect groups of blocks of profinite groups exist \cite[Theorem~5.2]{JohnRicardoBrauer1}, are unique up to conjugation \cite[Proposition~5.7]{JohnRicardoBrauer1}, and are open in a Sylow $p$-subgroup of the ambient group that contains them \cite[Proposition~5.8]{JohnRicardoBrauer1}. All statements (and their proofs) in this section go through without the assumption that $k$ be algebraically closed.

\begin{prop}[{\cite[Proposition~6.4]{JohnPeterBrauerTheory}}]\label{prop:idempotentsdirectlimit}
   For a commutative ring $R$, denote by $E(R)$ its set of primitive idempotents. Let $G$ be a profinite group.  For open normal subgroups $N\subseteq M$ of $G$, the maps
    $$\psi_{MN} : E(Z(\bigO[G/M]))\to E(Z(\bigO[G/N]))$$
    sending $b_M\in E(Z(\bigO[G/M]))$ to the unique $b_N\in E(Z(\bigO[G/N]))$ such that $\varphi_{MN}(b_N)b_M=b_M$, yield a direct system of finite sets, whose direct limit can be naturally identified with $E(Z(\bigO\db{G}))$.  In particular, $E(Z(\bigO\db{G}))$ is discrete with respect to the topology inherited from $\bigO\db{G}$.  The induced maps
     $$\psi_N : E(Z(\bigO[G/N]))\to E(Z(\bigO\db{G}))$$ send $b_N$ to the unique $b\in E(Z(\bigO\db{G})$ such that $\varphi_N(b)b_N = b_N$.
\end{prop}

Note that the previous proposition holds equally over $k$, but we will mostly work over $\bigO$ in this paper. 
We will call primitive idempotents in $Z(\bigO\db{G})$ ``block idempotents''. 

\begin{defn}
    With the notation from Proposition \ref{prop:idempotentsdirectlimit}, a \emph{compatible system of idempotents} $(b_N)_{N}$ is a collection of block idempotents $b_N\in Z(\bigO[G/N])$, where $N$ ranges over some fixed cofinal system of open normal subgroups of $G$, such that
    $$
        \psi_{MN}(b_{M})=b_N
    $$
    for all $N\subseteq M$ in the cofinal system.
    We call it a compatible system of idempotents \emph{for $b$} $\in E(Z(\bigO\db{G}))$ if $\psi_{N}(b_{N}) = b$ for some (and therefore every) $N$ in the cofinal system.
\end{defn}

We summarise below the consequences of Proposition~\ref{prop:idempotentsdirectlimit} that will be relevant to us. As standard, throughout the article, the subscripts ``$O$'' and ``$C$''  mean open and closed, respectively, so for instance $N\unlhd_O G$ means that $N$ is an open normal subgroup of $G$.

\begin{corol}\label{corol:blockidempotents}
    Let $G\unlhd_C H$ be profinite groups, and let $b\in Z(\bigO\db{G})$ be a block idempotent.
    \begin{enumerate}
        \item There exists a compatible system of idempotents for $b$  and the cofinal system may be chosen to be all open normal subgroups contained in some fixed $N_0\unlhd_O G$.
        \item The stabiliser of $b$ in $H$ is open, and we can choose the compatible system of idempotents from the first part so that $\Stab_H(b_N)=\Stab_H(b)$ for each $N$ in the cofinal system.
    \end{enumerate}
\end{corol}

\begin{proof}
    The first part is immediate from Proposition~\ref{prop:idempotentsdirectlimit}. For the second part, the fact that $\Stab_H(b)$ is open follows from the fact that the set of block idempotents is discrete in $\bigO\db{G}$ (upon which $H$ acts continuously by conjugation since $G$ is normal). Choose a compatible system of idempotents as in the first part, taking $N_0 = G\cap M_0$ for a small enough open normal subgroup $M_0$ of $H$, in order that $N_0$ be normal in $H$.  Because the direct system from Proposition \ref{prop:idempotentsdirectlimit} is $H$-equivariant, we have  
    $$\Stab_H(b) = \bigcup_{M\unlhd_O H} \Stab_H(b_{M\cap N_0}).$$
    But the stabilisers $\Stab_H(b_{M\cap N_0})$ are open, and hence we must have
    $\Stab_H(b)=\Stab_H(b_{M\cap N_0})$ for some $M$, proving the claim.
\end{proof}

We will now recall the definition and relevant properties of defect groups.

\begin{defn}
    Let $G$ be a profinite group and let $A$ be a pseudocompact $G$-algebra such that $1_A$ is a primitive idempotent in $A^G$. 
    \begin{enumerate}
    \item We define
    $$
        \oTr_H^G(A^H)= \bigcap_{N\unlhd_O G} \Tr_{NH}^G(A^{NH}).
    $$
    \item We call a closed subgroup $D\leq G$ a \emph{defect group} of $A$ if $1_A\in \oTr_D^G(A)$, and no proper closed subgroup of $D$ has the same property. 
    \end{enumerate}
\end{defn}

This is slightly more general than the definition given in \cite{JohnRicardoBrauer1}.
In particular, the paper \cite{JohnRicardoBrauer1} defines defect groups of block algebras only over $k$. However, the same definition can be used over $\bigO$ and Proposition~\ref{prop:defectgroupsoveroarethesame} below shows that, just like in the finite case, defect groups over $k$ and $\bigO$ coincide.

\begin{prop}\label{prop:defectgroupsoveroarethesame}
    Let $b\in Z(\bigO\db{G})$ be a block idempotent, where $G$ is a profinite group, and denote by $\bar b$ its image in $Z(k\db{G})$. Then the defect groups of $\bigO\db{G}b$ and $k\db{G}\bar b$ are the same.  
\end{prop}
\begin{proof}
    Pick some closed subgroup $H$ of $G$. We will prove the claim by showing that $b\in \oTr_H^G(\bigO\db{G}^H)$ if and only if $\bar b \in \oTr_H^G(k\db{G}^H)$.
    Clearly $b\in \oTr_H^G(\bigO\db{G}^H)$ implies $\bar b \in \oTr_H^G(k\db{G}^H)$. Conversely, assume $\bar b \in \oTr_H^G(k\db{G}^H)$. Take $N\unlhd_O G$, and take an $\bar x\in (k\db{G}b)^{HN}$ such that $\Tr_{HN}^G(\bar x)=\bar b$. The reduction map $\bigO\db{G}^{NH}\longrightarrow k\db{G}^{NH}$ is surjective. {To see this, note that it is the inverse limit of the reduction maps $\bigO[G/M]^{NH}\longrightarrow k[G/M]^{NH}$ ($M\unlhd_O G$), and these maps are clearly surjective given that domain and codomain are spanned as $\bigO$-modules by the $NH$-conjugacy class sums of group elements (and inverse limits in the category of pseudocompact $\bigO$-modules are exact by \cite[Th\'{e}or\`{e}me 3 in Chapter IV]{GabrielCategoriesAbeliennes}).}
    So we can pick an $x\in \bigO\db{G}^{NH}$ reducing to $\bar x$ such that $\Tr_{HN}^G(x)=b+y$ for some $y\in J(\bigO)Z(\bigO\db{G})$. Then $1+y$ is invertible in $Z(\bigO\db{G})$, giving us $\Tr_{HN}^G((1+y)^{-1}xb)=b$. Since this is true for all $N$ we get $b\in \oTr_H^G(\bigO\db{G}^H)$. 
\end{proof}

\begin{prop}\label{prop:defectgroupmaps}
    Let $G$ be a profinite group and let $A$ and $B$ be pseudocompact $G$-algebras. Assume that 
    $$
        \varphi:\ A \longrightarrow B
    $$
    is a continuous and additive $G$-equivariant map restricting to a ring isomorphism between $Z(A)$ and $Z(B)$. Given a primitive idempotent $e\in Z(A)$, any defect group of $Ae$ contains a defect group of $B\varphi(e)$.
\end{prop}
\begin{proof} 
    Let $D$ be a defect group of $Ae$. Then $e\in \Tr_{ND}^G(A^{ND})$ for any $N\unlhd_O G$. By equivariance, $\varphi(e)\in \Tr_{ND}^G(B^{ND})$ for any $N$, and therefore $\varphi(e) \in \oTr_D^G(B^D)$. Hence, there is a defect group of $B$ contained in $D$ (one may use Zorn's lemma as in \cite[Theorem~5.2]{JohnRicardoBrauer1} to find it).
\end{proof}

\begin{prop}\label{prop:defectgrpmodulopprime}
    Let $G$ be a profinite group and let $A$ be a pseudocompact $G$-algebra. If $N$ is a normal pro-$p'$ subgroup of $G$ acting trivially on $A$, then a closed subgroup $D$ of $G$ is a defect group of $A$ as a $G$-algebra if and only if $DN/N$ is a defect group of $A$ as a $G/N$-algebra.   
\end{prop}

\begin{proof}
    For any $H\unlhd_C G$, any $M\unlhd_O G$ and any $x\in A^{MH}=A^{MHN/N}$ we have $\Tr_{MH}^G(x)=[MHN:MH]\Tr_{MHN/N}^{G/N}(x)$. Since $[MHN:MH]\in\bigO^\times$ it follows that $\oTr_{H}^G(A^H)=\oTr_{HN/N}^{G/N}(A^{HN/N})$. The proposition now follows from the definition of a defect group.
\end{proof}

\section{Finiteness properties of blocks and defect groups}

In this section we will show that a block of a profinite group is isomorphic to a block of a virtually pro-$p$ group if and only if it has only finitely many simple modules, and we will show that every block with a topologically finitely generated defect group has this property. Given that the number of simple modules determines the number of vertices of the Gabriel quiver of the basic algebra of the block, a natural follow-up question is when that basic algebra is finitely presented. This is classically determined by $\Ext^1$'s and $\Ext^2$'s between simples, and we will show that this remains true in the pseudocompact setting.

\begin{theorem}\label{theorem:FiniteSimplesIFFBlockOfvProp}
    Let $\bigO\db{G}b$ be a block of a profinite group $G$.  Then $\bigO\db{G}b$ has finitely many simple modules if, and only if, it is a block of a continuous quotient group of $G$ that is virtually pro-$p$.
\end{theorem}

\begin{proof}
    Firstly, if $\bigO\db{G}b$ is isomorphic to a block of a virtually pro-$p$ group $H$, then it has finitely many simples, because $\bigO\db{H}$ has finitely many simples. The reason is that a simple pseudocompact module is automatically discrete and therefore a module over a finite quotient of $H$. The image of $O_p(H)$ becomes a normal $p$-subgroup in such a finite quotient and therefore acts trivially on any simple. That is, the simple $\bigO\db{H}$-modules are the simple $\bigO[H/O_p(H)]$-modules, of which there are only finitely many.  So we must show the other direction. 
    
    Suppose that $\bigO\db{G}b$ has finitely many simples.  Denote by $N$ the kernel of the continuous homomorphism $G\to (\bigO\db{G}b)^{\times}$ sending $g$ to $gb$.  Then $N$ acts trivially on every pseudocompact $\bigO\db{G}b$-module, because given such a module $V$, $n\in N$ and $v\in V$ we have $nv = nbv = bv = v$.  Furthermore, it is maximal as such, because any element not in $N$ acts non-trivially on $\bigO\db{G}b$ by construction.
    
    Then $\bigO \db{G}b$ is a block of $\bigO\db{G/N}$, because if $\bigO\db{G} = \bigO \db{G}b \times C$ then $\bigO\db{G/N} = \bigO \db{G}/I_N\bigO \db{G} = \bigO \db{G}b\times C/I_NC$, where $I_N$ denotes the augmentation ideal of $\bigO\db{N}$, and we have used $I_N\bigO\db{G}b = 0$. So it remains to check that $G/N$ is virtually pro-$p$.  

    As there are finitely many simples, each with open kernel, there is an open normal subgroup $L$ of $G$ that acts trivially on all the simple modules.  We claim that $LN/N$ is a pro-$p$ group, which will imply that $G/N$ is a virtually pro-$p$ group because $LN$ is open in $G$.

    If $LN/N$ is not pro-$p$, then there is a pro-$p'$ element $g$ of $L$ not in $N$ -- this can be seen by taking a non-pro-$p$ element $g'$ and writing the abelian group $\overline{\langle g'\rangle}$ as a product of a pro-$p$ group and a pro-$p'$ group.  We will obtain a contradiction by showing that $g\in N$.  Let $H = \overline{\langle g\rangle}$.  We claim that $H$ acts trivially on every pseudocompact $\bigO\db{G}b$-module, so let $U$ be such a module.  Then $g$ acts trivially on $U$ if and only if it acts trivially on every finite length quotient of $U$, so we will suppose that $U$ is of finite length and work by induction, with the base case coming by our hypothesis on $L$.  If $U$ is not simple, then there is a short exact sequence of $\bigO\db{G}$-modules
    $$0 \to V\to U\to W\to 0,$$
    and $H$ acts trivially on $V,W$. Since these are modules of finite length, $H$ acts on all three modules above through some finite quotient (necessarily of $p'$-order). Since $V$ and $W$ lie in the principal block of this finite quotient, so does $U$ and, being a $p'$-group, the finite quotient acts trivially on all modules in its principal block. So $g$ acts trivially on $U$, and therefore on all pseudocompact $\bigO\db{G}b$-modules. 
    The same is true of all conjugates of $g$ because $L$ is normal, and hence the normal subgroup generated by $g$ is contained in $N$, because $N$ is the unique maximal normal subgroup acting trivially on every pseudocompact $\bigO\db{G}b$-module. 
\end{proof}

    \begin{theorem}\label{theorem:reduction_to_virtually_pro_p}
        Let $\bigO\db{G}b$ be a block with defect group $D\leq G$.
        If $D$ is topologically finitely generated, then there is a closed normal pro-$p'$ subgroup 
        $N_0\unlhd G$ such that for any closed normal subgroup $N$ of $G$ which is open in $N_0$, the quotient $G/N$ is virtually pro-$p$ and $\varphi_N$ restricts to an isomorphism of pseudocompact algebras
        $$\bigO\db{G} b \stackrel{\sim}{\longrightarrow} \bigO\db{G/N}\varphi_N(b).$$
        Furthermore, the defect group of $\bigO\db{G/N}\varphi_N(b)$ is isomorphic to $D$.
    \end{theorem}

    \begin{proof}
        Pick a Sylow $p$-subgroup $S\leq G$ containing $D$. Since $D$ has finite index in $S$, the group $S$ is also topologically finitely generated. By \cite[Corollary 2.2.2]{ColinReidThesis} (or \cite[Corollary~2.4.7]{CapraceReidEtAL}) this implies that the quotient $G/O_{p'}(G)$ is virtually pro-$p$. We also know that there is an $M\unlhd_O G$ such that $\varphi_M:\ \bigO \db{G} \twoheadrightarrow \bigO [G/M]$ does not map $b$ to zero. We define $N_0=M\cap O_{p'}(G)$ and let $N$ be any closed normal subgroup of $G$ which is open in $N_0$. It is clear that $G/N$ is virtually pro-$p$.  Furthermore, since $N$ is a pro-$p'$ subgroup of $G$, the natural map $\varphi_N$ restricts to an isomorphism $\bigO \db{G}b_N\longrightarrow \bigO \db{G/N}$, where $b_N$ denotes the principal block idempotent of $\bigO \db{N}$ (this follows from the analogous claim for finite groups, which is elementary -- a proof will appear in \cite{JohnPeterExtendedBrauerUnfinished}). As $b_N$ is clearly $G$-stable it follows that $b_N\in Z(\bigO \db{G})$ and therefore either $b_Nb=b$ or $b_N b=0$. In the latter case the image of $b$ in $\bigO\db{G/N}$ would be zero, but we made sure that it is not. So $b_Nb=b$ and $\varphi_N$ restricts to an isomorphism $\bigO\db{G} b \longrightarrow \bigO\db{G/N}\varphi_N(b)$, as required.
        
        Since the isomorphism
        $
         \bigO\db{G}b \longrightarrow \bigO\db{G/N}\varphi_N(b) 
        $
        is clearly $G$-equivariant if we let $G$ act by conjugation, we can apply Proposition~\ref{prop:defectgroupmaps} to it and its inverse. It follows that $\bigO\db{G}b$ and $\bigO\db{G/N}\varphi_N(b)$ both have defect group $D$ as a $G$-algebra. Using Proposition~\ref{prop:defectgrpmodulopprime} it follows that $\bigO\db{G/N}\varphi_N(b)$ has defect group $DN/N\cong D$ as a $G/N$-algebra. This completes the proof.
    \end{proof}

    Combining the two previous theorems, one obtains:

    \begin{corol}\label{corol:D_fg_implies_finite_simples}
        Let $\bigO\db{G}b$ be a block with defect group $D\leq G$.
        If $D$ is topologically finitely generated, then $\bigO\db{G}b$ has only finitely many isomorphism classes of simple modules.
    \end{corol}

\begin{remark}\label{remark:vpropeasier}
     Being precise about which blocks of profinite groups are blocks of virtually pro-$p$ groups may have representation theoretic applications, because the modular representation theory of virtually pro-$p$ groups seems to be considerably more tractable than the representation theory of arbitrary profinite groups -- cf.\ for instance \cite{JohnModReps, JohnGreenCorr} where basic results like Green's Indecomposability Theorem and Green Correspondence are proved only for virtually pro-$p$ groups.  The same sort of situation will manifest itself in Theorem \ref{thm:normal-defect}, whose proof works only for virtually pro-$p$ groups.  Theorem \ref{theorem:reduction_to_virtually_pro_p} says that this does not limit the generality of our main Theorem \ref{thm:reduction} and its corollary: that Donovan's Conjecture holds for blocks of profinite groups with defect group $\Z_p^n$.
\end{remark}

\begin{defn}
    We say a basic pseudocompact $k$-algebra is \emph{finitely presented} if it is isomorphic to 
    $k\db{Q}/I$ for a finite quiver $Q$ and a finitely generated relation ideal $I$.
\end{defn}

We will now provide a sufficient condition for finite presentability of the basic algebra of a block. We do this over $k$. This will involve reproving some elementary results on finite-dimensional algebras in the pseudocompact setting. Whenever it matters, we will be explicit about the fact that we are working in the category $\PC(A)$ of pseudocompact modules over a pseudocompact ring $A$. In some instances it does not matter, though, such as for finite generation of modules and ideals. 
Also, by \cite[Proposition~3.2]{IusenkoMacQuarrieSemisimple} the ``topological Jacobson radical'' $J(A)$ of a pseudocompact $k$-algebra $A$ is the same as its Jacobson radical as a $k$-algebra. In cases where it does make a difference, we will either specify the category or add the adverb ``topologically''. For instance, ``topologically hereditary'' means that all \emph{closed} submodules of projective modules are projective.  

We require the following result only for finite quivers, but since it is true for arbitrary quivers (see e.g.\ \cite[Definition 2.1]{IusenkoMacQuarrieExtensions} for the definition of $k\db{Q}$), we state it in this greater generality.

\begin{prop}\label{prop:kQhereditary}
    Let $Q$ be a quiver. Then $k\db{Q}$ is topologically hereditary.
\end{prop}

\begin{proof}
This can be obtained by duality from the analogous result for coalgebras \cite[Theorem 4]{ChinPathCoalgebras}, but we give a direct proof here for the reader's convenience.

    Just as with finite-dimensional algebras, the global dimension of $k\db{Q}$ is the projective dimension of the topologically semisimple $k\db{Q}$-module $\Sigma = k\db{Q}/J(k\db{Q})$.  We have a short exact sequence 
    $$
        0\longrightarrow J(k\db{Q}) \longrightarrow k\db{Q} \longrightarrow \Sigma \longrightarrow 0,
    $$
    so we need only check that $J(k\db{Q})$ is projective as a left $k\db{Q}$-module.  We have $k\db{Q} = \invlim k\db{R}$ where $R$ runs through the finite subquivers of $Q$, and $J(k\db{Q}) = \invlim J(k\db{R})$.  So by \cite[Corollary 3.3]{Brumer}, it is enough to check our claim supposing that $Q$ is finite.  
    
    In this case, as a left $k\db{Q}$-module, $J(k\db{Q}) = \bigoplus k\db{Q}\alpha$ as $\alpha$ runs through the arrows of $Q$.  But $k\db{Q}\alpha$ is projective, because if $e$ is the vertex idempotent such that $e\alpha = \alpha$, the map $k\db{Q}e\to k\db{Q}\alpha$ sending $x$ to $x\alpha$ is an isomorphism of left $k\db{Q}$-modules. \qedhere
\end{proof}

\begin{prop}\label{prop:ext1}
    Let $A$ be a pseudocompact ring. Then
    $$
        \Ext^1_{\PC(A)}(A/J(A), A/J(A)) \cong \Hom_{\PC(A)}(J(A)/J^2(A), A/J(A)).
    $$
\end{prop}
\begin{proof}
    After applying $\Hom_{\PC(A)}(-,A/J(A))$ to the short exact sequence 
    $0\longrightarrow J(A)\longrightarrow A\longrightarrow A/J(A)\longrightarrow 0$ we get the following (part of a) long exact sequence:
    $$
        \Hom_{\PC(A)}(A, A/J(A))\stackrel{0}{\longrightarrow} \Hom_{\PC(A)}(J(A),A/J(A))\longrightarrow \Ext^1_{\PC(A)}(A/J(A),A/J(A))\longrightarrow 0.
    $$ 
    So $\Ext^1_{\PC(A)}(A/J(A), A/J(A)) \cong \Hom_{\PC(A)}(J(A),A/J(A)) \cong \Hom_{\PC(A)}(J(A)/J^2(A),A/J(A))$.
\end{proof}

The $A$-module $A/J(A)$ is finitely generated, so every short exact sequence of $A$-modules
$$0 \longrightarrow A/J(A)\longrightarrow E\longrightarrow A/J(A)\longrightarrow 0$$
is continuous.  But this does \emph{not} imply that $\Ext^1_{\PC(A)}(A/J(A), A/J(A))$ is the same as the corresponding $\Ext^1$ of abstract modules.  For example, consider as in \cite[Example 2.14]{IusenkoMacQuarrieSemisimple} the power series algebra in countably many variables $A = k\db{x_1, x_2, \hdots}$ and the element $y = \sum_{i=1}^{\infty}{x_i}^2 \in J^2\setminus J\cdot J$, where $J\cdot J$ denotes the abstract product of $J = J(A)$ with itself and $J^2$ denotes its topological closure.  Let $I$ be a vector space complement of $\langle y\rangle$ in $J$ containing $J\cdot J$. Then $A/I$ is a two-dimensional $A$-module that is not pseudocompact: the closure of $I$ contains $y$, and hence is $J$.  It is thus an abstract extension of $k$ by $k$ that is not in $\Ext_{\PC(A)}^1(k,k)$.  However, if $A/J^2(A)$ is finite dimensional, then the abstract and pseudocompact $\Ext^1(A/J,A/J)$'s coincide, because in this case $A/J$ has a projective presentation by finitely generated projective modules, which is thus also a projective presentation of $A/J$ by abstractly projective modules, so the $\Ext^1$ groups are the same (cf.\ \cite[Lemma 2.4]{JohnPavelPeterInfGen}).

\begin{prop}[Criterion for finite presentability]\label{prop:finitelypresented}
    Let $A$ be a basic pseudocompact $k$-algebra. 
     Then $A\cong k\db{Q}/I$ for a \emph{finite} quiver $Q$ and some closed relation ideal $I$ if and only if 
     $$\dim_k (A/J(A))<\infty\quad\textrm{and}\quad\dim_k(\Ext^1_{\PC(A)}(A/J(A),A/J(A)))<\infty.$$ 
    Furthermore, $A$ is finitely presented if and only if, in addition, $\dim_k(\Ext^2_{\PC(A)}(A/J(A),A/J(A)))<\infty$.
\end{prop}
\begin{proof}
    By Proposition~\ref{prop:ext1}, $\Ext^1_{\PC(A)}(A/J(A),A/J(A))$ being finite-dimensional implies that $J(A)/J^2(A)$ is finite-dimensional. If $A/J(A)$ and $J(A)/J^2(A)$ are both finite-dimensional, then by \cite[Section 6]{KrauseVossieck} we have $A\cong k\db{Q}/I$ for some finite quiver $Q$ and a closed ideal $I\subseteq J^2(k\db{Q})$. Conversely, if $A\cong k\db{Q}/I$ for finite $Q$ and a closed ideal $I$, then $A/J(A)$ and $J(A)/J^2(A)$ are finite-dimensional, so by Proposition~\ref{prop:ext1}, so is $\Ext^1_{\PC(A)}(A/J(A),A/J(A))$. This proves the first assertion.

    For the second part, we adapt an argument from \cite{Bongartz} (attributed to Butler). Assume that $A=k\db{Q}/I$ for a finite quiver $Q$ and a closed ideal $I\subseteq J^2$, where $J=J(k\db{Q})$. We need to show that $I$ is finitely generated as a two-sided ideal if and only if $\Ext^2_{\PC(A)}(A/J(A),A/J(A))$ is finite-dimensional.  The sequence of inclusions $I^2J\subseteq I^2\subseteq IJ\subseteq I\subseteq J$ yields naturally the sequence of pseudocompact $A$-modules
    $$
        IJ/I^2J \stackrel{\alpha}{\longrightarrow} I/I^2 \stackrel{\beta}{\longrightarrow} J/IJ \longrightarrow k\db{Q}/I \longrightarrow k\db{Q}/J\longrightarrow 0,
    $$
    which is exact by construction.
    Since $k\db{Q}$ is topologically hereditary (Proposition \ref{prop:kQhereditary}), $J$, $I$ and $IJ$ are projective pseudocompact $k\db{Q}$-modules.  If $P$ is a projective $k\db{Q}$-module, then $P/IP$ is a projective $k\db{Q}/I$ module, hence $J/IJ$, $I/I^2$ and $IJ/I^2J$ are projective pseudocompact $A = k\db{Q}/I$-modules. So this is the start of a projective resolution of $k\db{Q}/J$ as a left $A$-module. We can apply $\Hom_{\PC(A)}(-, k\db{Q}/J)$ to it, and, given that $I\subseteq J^2$ implies $\beta^*=0$, we have
    $$
        \Ext^2_{\PC(A)}(k\db{Q}/J,k\db{Q}/J) = \Ker(\alpha^*)= \Hom_{\PC(A)}(I/(IJ+JI), k\db{Q}/J).
    $$  
    The vector space $\Hom_{\PC(A)}(I/(IJ+JI), k\db{Q}/J)$ is finite-dimensional if and only if $I/(IJ+JI)$ is. Clearly, if $I$ is finitely-generated as a $k\db{Q}$-bimodule, then $I/(IJ+JI)$ is finitely-generated as a $k\db{Q}/J$-bimodule, making it finite-dimensional. Conversely, if $I/(IJ+JI)$ is finite-dimensional, then let $x_1,\ldots,x_n\in I$ be the preimages of a basis. Define $M=I/(k\db{Q}x_1k\db{Q}+\ldots+k\db{Q}x_nk\db{Q})$ so that $M/(JM+MJ)=0$. We now simply use Brumer's version of Nakayama's lemma \cite[Lemma~1.4]{Brumer} twice: $M/(JM+MJ)=0$ implies that $M/MJ=0$, which in turn implies that $M=0$. That is, $I$ is finitely generated.
\end{proof}

\begin{theorem}
    Let $k\db{G}b$ be a block of a profinite group $G$ with defect group $D$, and let $A$ denote a basic algebra of $k\db{G}b$. 
    \begin{enumerate}
        \item If $D$ is finitely generated as a pro-$p$ group, then $A\cong k\db{Q}/I$ for a finite quiver $Q$ and a closed relation ideal $I$.
        \item If $D$ is finitely presented as a pro-$p$ group, then $A$ is finitely presented.
    \end{enumerate} 
\end{theorem}

\begin{proof}
    Write $A=k\db{G}b$ and $J=J(k\db{G}b)$.
    By Theorem~\ref{theorem:reduction_to_virtually_pro_p} we can assume that $G$ is virtually pro-$p$, and by Corollary \ref{corol:D_fg_implies_finite_simples}, $A/J$ is finite-dimensional. Let $S$ denote a Sylow $p$-subgroup of $G$ containing $D$. If $D$ is finitely generated then so is $S$, since $D$ has finite index in it. 
    If $D$ is finitely presented, then so is the open normal subgroup $\bigcap_{g\in G} D^g$ by \cite[Proposition~12.2.1]{WilsonProfiniteGroups},
    so that $S$ is finitely presented by \cite[Proposition~12.2.2]{WilsonProfiniteGroups}.
    Now \cite[Theorems~7.8.1~and~7.8.3]{RibesZalesskiiProfiniteGroups} imply that if $S$ is finitely generated then  $H^1(S, \F_p)$ is finite and if $S$ is finitely presented then $H^2(S, \F_p)$ is finite as well. 
    
    Fix some $i>0$ and assume that $H^i(S,\F_p)$ is finite. We have
    $$
        H^i(S,k) = \varinjlim_{U\unlhd_O S} H^i(S/U, k) \cong  \varinjlim_{U\unlhd_O S} H^i(S/U, \F_p)\otimes k \cong (\varinjlim_{U\unlhd_O S} H^i(S/U, \F_p))\otimes k \cong H^i(S,\F_p)\otimes k,
    $$
    where we have used \cite[Proposition~8 of Chapter 1]{SerreGaloisCohomology} to write $H^i(S,k)$ as a direct limit, as well as the fact that tensor products commute with direct limits. Of course, for the latter we also need that the isomorphism $H^i(S/U, \F_p)\otimes k\longrightarrow H^i(S/U, k):\ [\alpha]\otimes x \mapsto [x\alpha]$ commutes with the maps of the inverse systems, which is elementary if one represents cocycles as maps from $S/U\times \ldots \times S/U$ into $\F_p$ or $k$. It follows that $H^i(S,k)\cong \Ext^i_{\PC(k\db{S})}(k,k)$ is finite-dimensional.

    For any short exact sequence $0\to M \to N \to K \to 0$ of finite-dimensional $k\db{S}$-modules and any finite-dimensional $k\db{S}$-module $L$ we get exact sequences 
    $$
        \Ext^i_{\PC(k\db{S})}(L, M) \to \Ext^i_{\PC(k\db{S})}(L, N) \to \Ext^i_{\PC(k\db{S})}(L, K)
    $$
    and 
    $$
     \Ext^i_{\PC(k\db{S})}(K, L) \to \Ext^i_{\PC(k\db{S})}(N, L)\to \Ext^i_{\PC(k\db{S})}(M, L).
    $$
    In particular, the middle terms are finite-dimensional if the outer terms are. By induction over the length of a module, we get that $\Ext^i_{\PC(k\db{S})}(M,N)$ is finite-dimensional for any two finite-dimensional $k\db{S}$-modules $M$ and $N$.

    Denoting as usual by $\Res^G_S$ the restriction of a $k\db{G}$-module to $k\db{S}$, it now follows that $\Ext^i_{\PC(k\db{S})}(\Res^{G}_S A/J, \Res^G_S A/J)$ is finite-dimensional. By \cite[Lemma~4.7]{Brumer} we have an embedding 
    $$
        \Ext^i_{\PC(A)}(A/J,A/J)\hookrightarrow \Ext^i_{\PC(k\db{S})}(\Res^G_S A/J, \Res^G_S A/J)
    $$
    and thus $\Ext^i_{\PC(A)}(A/J,A/J)$ is finite-dimensional. We can now apply Proposition~\ref{prop:finitelypresented}, using that the above implies that $\dim_k(\Ext^1_{\PC(A)}(A/J,A/J))<\infty$ if $D$ is finitely generated and $\dim_k(\Ext^2_{\PC(A)}(A/J,A/J))<\infty$ if $D$ is finitely presented.
\end{proof}

\section{Morita reductions and blocks with defect $\Z_p^n$}\label{section:moritareductions}

In this section we will generalise two well-known reductions for Donovan's conjecture to profinite groups. These reductions work for arbitrary blocks of profinite groups with arbitrary defect groups. However, our concrete application to blocks with defect group $\Z_p^n$ relies on the results of the previous section which reduce to the case of blocks of virtually pro-$p$ groups. Lemma~\ref{lemma:group-theory-caseAB} below is the key group theoretic observation that makes it possible to reduce to blocks with normal defect groups in that setting. Unlike in the finite case, the classification of finite simple groups plays no role whatsoever in the proof of Donovan's Conjecture for blocks with defect group $\Z_p^n$ (Corollary \ref{corol:DonovanforZpn}). The arguments resemble those used for blocks of finite $p$-solvable groups \cite{KuelshammerpSolvable}, rather than those used for blocks with finite abelian defect groups. In fact, finite abelian defect groups are much more difficult: Donovan's Conjecture is known to hold only for $p = 2$ (by \cite[Theorem~1.7]{EEL}) and for cyclic defect groups (by Brauer and Dade, cf. \cite[Chapter 11]{LinckelmannBookVol2}), and is only known to hold without explicit use of the classification of finite simple groups for blocks with cyclic or Klein four defect group (the latter being due to Erdmann \cite{ErdmannV4} and Linckelmann \cite{LinckelmannV4}). 

\begin{lemma}\label{lemma:group-theory-caseAB}
    Let $G$ be a virtually pro-$p$ group with an open subgroup $D\cong \Z_p^n$ for some $n\in \N$. Assume that $O_p(G)\subseteq D$. Then one of the following two cases occurs:
    \begin{enumerate}
        \item[(A)] There is a normal subgroup $N\unlhd G$ such that 
            \begin{enumerate}
                \item $N\cap D^g=1$ for all $g\in G$ (in particular $|N|<\infty$), and
                \item $N \supsetneqq O_{p'}(Z(G))$.
            \end{enumerate}
        \item[(B)] $D$ is normal in $G$ and $C_G(D)=D O_{p'}(Z(G))$.
    \end{enumerate}
\end{lemma}
\begin{proof}
    Firstly note that $O_p(G)$ is an open subgroup of $D\cong \Z_p^n$, and therefore $O_p(G)\cong \Z_p^n$.
    We consider the group $H=C_G(O_p(G))$. Note that $O_p(G)$ is a central subgroup of finite index in $H$, and therefore \cite[Theorem~5.6]{isaacs2008finite} implies that the transfer map
    \begin{equation}
        v:\ H \longrightarrow O_p(G):\ g \mapsto g^{[H:O_p(G)]}
    \end{equation}
    is a group homomorphism. Its kernel is a normal subgroup of $H$ and, a fortiori, of $G$ (since both $H$ and $O_p(G)$ are characteristic in $G$). Write $N=\Ker(v)$. Straight from the definition of $v$, $N\cap D^g = 1$ for every $g$, so in particular $N$ is finite.  The normal subgroup $O_{p'}(Z(G))$ of $G$ is clearly contained in $H$, and is finite because $G$ is virtually pro-$p$, so $O_{p'}(Z(G))\subseteq N$ because the codomain of $v$ is torsion-free.
    
    If $O_{p'}(Z(G))\subsetneqq N$, then we have already checked that the conditions of (A) are satisfied, so it remains to consider the case where $O_{p'}(Z(G)) = N$.  The image of $v$ is open in $O_p(G)\cong \Z_p^n$, so isomorphic to $\Z_p^n$, and hence $H$ is a central extension of $\Z_p^n$ by $O_{p'}(Z(G))$.  If $Q$ is a Sylow $p$-subgroup of $H$, then $H = QN$ (by \cite[Theorem 2.3.15]{RibesZalesskiiProfiniteGroups} for instance) so $Q$ is normal (because $N$ is central), hence characteristic.  So $Q\unlhd G$ and hence $Q\subseteq O_p(G)\subseteq D$.  On the other hand, it is clear that $D\subseteq Q$ because $D$, being abelian, is a pro-$p$ group that centralises $O_p(G)\subseteq D$. So $D=Q=O_p(G)$ and hence $D$ is normal in $G$.  Finally,
    \begin{displaymath}
    C_G(D) = H = QN = DO_{p'}(Z(G)).\qedhere
    \end{displaymath}
\end{proof}

Recall that if $A$ is an algebra over a commutative ring and $e\in A$ is an idempotent such that $AeA=A$, then $A$ and $eAe$ are Morita equivalent (see, for instance, \cite[Theorem~2.8.7]{LinckelmannBookVol1}). An idempotent $e$ with this property is called \emph{full}. Now consider a \emph{pseudocompact} algebra $A$ and a full idempotent $e\in A$. By an elementary calculation, the condition $AeA=A$ implies that $eA$ and $Ae$ are finitely generated as left and right $eAe$-modules, respectively, and therefore \cite[Proposition~2.2]{JohnPavelPeterInfGen} implies that $Ae\otimes_{eAe}eA \cong Ae\widehat \otimes_{eAe}eA$ and $eA\otimes_A Ae \cong eA\widehat \otimes_A Ae$. Thus we get isomorphisms of pseudocompact bimodules $Ae\widehat\otimes_{eAe} eA\cong A$ and $eA\widehat\otimes_A Ae\cong eAe$, both given by multiplication, which implies that the functors $eA\widehat{\otimes}_A -$ and $Ae\widehat{\otimes}_{eAe} -$ induce mutually inverse equivalences between the categories of pseudocompact modules $\PC(A)$ and $\PC(eAe)$. That is, both the ordinary module categories and the categories of pseudocompact modules are equivalent. All Morita equivalences in this paper will be of this form, that is, they are Morita equivalences of pseudocompact algebras in the appropriate sense.

\begin{prop}[{Profinite Fong reduction, cf. \cite[Theorem~6.8.3]{LinckelmannBookVol2}}]\label{prop:fong1}
    Let $G$ be a profinite group with a closed normal subgroup $N$. Let $b\in Z(\bigO \db{G})$ and $c\in Z(\bigO \db{N})$ be block idempotents with $bc\neq 0$. Let $H$ denote the stabiliser of $c$ in $G$. Then
    \begin{enumerate}
        \item There is a unique block idempotent $d\in Z(\bigO\db{H})$ such that $b=\Tr_H^G (d)$. We have $bd=d=cd$.
        \item $\bigO\db{G}b$ and $\bigO\db{H}d$ are Morita equivalent. More precisely, $d \bigO\db{G}b d= \bigO\db{H}d$ and $d$ does not annihilate any simple $\bigO\db{G}b$-module.
        \item Every defect group of $\bigO\db{H}d$ is a defect group of $\bigO\db{G}b$.
    \end{enumerate}
\end{prop}

\begin{proof}
    Given idempotents $x,y$ we use the notation $xy$ or $x\cdot y$ inconsistently, in order to make $x\cdot{^g}y$ unambiguous. Recall first that $H$ is open in $G$ by Corollary \ref{corol:blockidempotents}.  The auxiliary \cite[Lemma 6.8.4]{LinckelmannBookVol2}, whose proof goes through verbatim, says that the maps $\tn{Tr}_H^G$ and ``multiplication by $c$'' yield mutually inverse algebra isomorphisms 
    $$Z(\bigO\db{H}c)\iso Z(\bigO\db{G}\tn{Tr}_H^G(c)).$$  
    The element $b\tn{Tr}_H^G(c) = \tn{Tr}_H^G(bc)$ is a non-zero central idempotent of $\bigO\db{G}$ because $bc\neq 0$, so $b$ being central primitive implies that  $b = b\tn{Tr}_H^G(c)$ -- that is, that $b$ is an element of the right hand side of the above isomorphism, and hence there is a unique $d\in Z(\bigO\db{H}c)$ with $\tn{Tr}_H^G(d) = b$, and $cd=dc=d$ because $d\in Z(\bigO\db{H}c)$.  Finally, since $c\cdot {^g}c = 0$ when $g\notin H$, we have
    $$bd = \Tr_H^G(d)d = \sum_{gH\in G/H} {^g}d\cdot{^g}ccd = cd = d.$$
    For the claim about Morita equivalence, note that $bd=d$, and that $d$ lies in the centre of $\bigO\db{H}$. We therefore have the following equality:
     $$
        d\bigO\db{G}bd=d\bigO\db{G}d=\bigoplus_{gH\in G/H} dg \bigO\db{H}d = \bigoplus_{gH\in G/H} d\cdot{^gd}g \bigO\db{H}= \bigO\db{H}d.
     $$
     To show fullness of $d$, note that ${^gd}\in \bigO \db{G} d \bigO \db{G}$ for all $g\in G$, and therefore $b=\Tr_H^G(d)\in \bigO \db{G} d \bigO \db{G}$. {On the other hand, $d\in d\bigO\db{G} = db\bigO\db{G}\subseteq b\bigO\db{G}$.} Hence,
     $\bigO \db{G} d \bigO \db{G} = \bigO \db{G} b$. It follows that $\bigO \db{G}b$ and $\bigO \db{H} d$ are Morita equivalent.
     
     To prove part (3), we use limits.  Choose an open normal subgroup $M\subseteq H$ of $G$ such that $0\neq \varphi_M(bc) = \varphi_M(b)\varphi_M(c)$, so there are block idempotents $b_M\in Z(\bigO [G/M]), c_M\in Z(\bigO[NM/M])$ such that $\varphi_M(b)b_M =b_M,\varphi_{M}(c)c_M =c_M$ and $b_Mc_M\neq 0$. By Corollary~\ref{corol:blockidempotents}, we can choose $M$ such that $H/M$ is the stabiliser of $c_M$ in $G/M$.  Consider the compatible systems of idempotents $(b_K)_K, (c_K)_K$ for $b,c$ determined by $b_M, c_M$, as $K$ runs through the open normal subgroups of $G$ contained in $M$ (see Corollary~\ref{corol:blockidempotents}).  For each $K$, $b_Kc_K\neq 0$ because $\varphi_{KM}(b_Kc_K)\neq 0$, so the conditions of the {finite version of this proposition (\cite[Theorem~6.8.3]{LinckelmannBookVol2})} apply: there is a unique block idempotent $d_K\in Z(\bigO[H/K]c_K)$ such that $\tn{Tr}_H^G(d_K) = b_K$.  The $d_K$ all have the same $G$-stabiliser $H$, so
     $$b_M\tn{Tr}_H^G(\varphi_{MK}(d_K)) = b_M\varphi_{MK}\tn{Tr}_H^G(d_K) = b_M\varphi_{KM}(b_K) = b_M.$$
     As the unique block idempotent $e$ of $\bigO[H/M]c_M$ such that $b_M\tn{Tr}_H^G(e) = b_M$ is $d_M$, it follows that $d_M\varphi_{MK}(d_K) = d_M$, and hence the system $(d_K)_K$ is compatible.  The same argument, applied to $d$ instead of $d_K$, shows that $(d_K)_K$ is a compatible system for $d$.

    Let $D$ be a defect group of $\bigO\db{G}b$ and $Q$ be a defect group of $\bigO\db{H}d$.
    By \cite[Corollary 5.10]{JohnRicardoBrauer1} there is an open normal subgroup $K_0\subseteq M$ of $G$ such that, for any normal subgroup $K\subseteq K_0$ of $G$, $DK/K$ is a defect group of $\bigO[G/K] b_K$ and $QK/K$ is a defect group of $\bigO[H/K] d_K$. By the finite version of the result \cite[Theorem~6.8.3]{LinckelmannBookVol2}, $DK/K$ and $QK/K$ are conjugate within $G$. So we can define $\emptyset \neq X_K=\{ g\in G \ | \ {^gDK}=QK \}$. The $X_K$ are non-empty clopen subsets of $G$ such that if $L\subseteq K$ then $X_L\subseteq X_K$, which by the standard compactness argument implies that their intersection is non-empty. An element $g$ in the intersection satisfies ${^gD}=Q$. \qedhere
\end{proof}

\begin{prop}\label{prop:defectgroupintersection}
    Let $G$ be a profinite group and $N$ a closed normal subgroup of $G$.  If $b$ is a block idempotent of $G$ and $c$ is a block idempotent of $N$ with $bc\neq 0$, then for any defect group $Q$ of $\bigO\db{N}c$, there is a defect group $D$ of $\bigO\db{G}b$ with $D\cap N = Q$.
\end{prop}

\begin{proof}
    By Proposition~\ref{prop:fong1} it is enough to prove the result supposing that $c$ is $G$-stable, and hence that $bc = b$.  We may work in a cofinal system of open normal subgroups $M$, together with compatible block idempotents (in the sense of Corollary~\ref{corol:blockidempotents})  $b_M\in Z(\bigO[G/M])$ and $c_M \in Z(\bigO[NM/M])$, such that $b_Mc_M\neq 0$, and if $\bigO\db{G}b$ has defect group $D$ and $\bigO\db{N}c$ has defect group $Q$, then $\bigO[G/M]b_M$ has defect group $DM/M$ and $\bigO[NM/M]c_M$ has defect group $QM/M$.

    Then \cite[Theorem~6.8.9~(i)]{LinckelmannBookVol2} (the finite version of this result) implies that for any $M$ in our cofinal system the set
    $$ X_M= \{g\in G \ | \ {^gDM}\cap NM = QM \} $$
    is non-empty. By the usual compactness argument, the intersection of the $X_M$ is non-empty, and any element $g$ in this intersection satisfies ${^gD}\cap N = Q$ (this uses $ \bigcap_M{^gDM}\cap NM={^gD}\cap N$).
\end{proof}

For the definitions and basic properties of 2-cocycles and central extensions for finite groups, see for instance \cite[Section 1.2 of Chapter 1]{LinckelmannBookVol1}.

\begin{defn}
    Let $G$ be a profinite group and let $\alpha\in Z^2(G/N, \bigO^\times)$ for some $N\unlhd_O G$. Given an open normal subgroup $M\subseteq N$ of $G$, the map $G/M\times G/M\to \bigO^{\times}$ given by $(gM,hM)\mapsto \alpha(gN,hN)$ is a 2-cocycle, {by slight abuse of notation also denoted $\alpha$}, and the natural projection $G/M\to G/N$ defines a homomorphism of twisted group algebras $\bigO_{\alpha}[G/M]\to \bigO_{\alpha}[G/N]$.  We define the \emph{{completed} twisted group algebra}
    $$
        \bigO_\alpha\db{G} = \varprojlim_{M} \bigO_\alpha[G/M],
    $$ 
    where $M$ ranges over the cofinal system formed by the normal subgroups of $G$ contained in $N$.
\end{defn}

\begin{prop}\label{prop:centralextension}
     Let $G$ be a profinite group and let $\alpha\in Z^2(G/N, \bigO^\times)$ for some $N\unlhd_O G$. Assume that the image of $\alpha$ consists of $p'$-roots of unity. Then there is a central extension of profinite groups
     $$
        1 \longrightarrow Z \longrightarrow L \longrightarrow G \longrightarrow 1
     $$
     for a finite abelian $p'$-group $Z$ such that $ \bigO\db{L}e\cong \bigO_\alpha\db{G}$ for an idempotent $e\in Z(\bigO\db{L})$. This is an isomorphism of $G$-algebras.
\end{prop}
\begin{proof}
    The proof is just as in the finite case.
    Via the isomorphism {$\bigO^{\times}\iso k^{\times}\times (1+J(\bigO))$}, the image of $\alpha$ is a finite subset of $k^{\times}$, so generates a finite $p'$-subgroup of $\bigO^\times$. Let $Z$ be this subgroup, and define $L=Z\times G$ with multiplication $(z,g)(z',g') = (zz'\alpha(g,g'),gg')$. The topology on $L$ is clearly profinite, and the multiplication is clearly continuous. So $L$ is a profinite group, and projection onto $G$ is a continuous group homomorphism.
    The inclusion of $Z$ in $\bigO^\times$ defines a character $\chi$ of $Z$.
    We set $e=e_{\chi}$, the central primitive idempotent of $\bigO\db{Z}$ associated to $\chi$. Since $Z$ is central in $L$, we have $e\in Z(\bigO{\db L})$, and one checks just as in the finite case that the map that sends $(z,g)e_\chi=(1,g)\chi(z)e_{\chi}$ (for $(z,g)\in L$) to $g\in \bigO_{\alpha}\db{G}$ is an algebra isomorphism. 
    Moreover, this isomorphism is clearly $L$-equivariant, that is, it is an isomorphism of $L$-algebras. Since the conjugation action of $Z$ is trivial on both domain and range, this isomorphism is an isomorphism of $G$-algebras ($G$ being $L/Z$).   
\end{proof}

\begin{remark}\label{remark:twistedgroupalgebra}
    In the situation of the foregoing proposition, let $\pi:\ L\twoheadrightarrow G$ denote the natural epimorphism. By Proposition~\ref{prop:defectgrpmodulopprime}, a closed pro-$p$ subgroup $D\leq L$ is a defect group for a block $\bigO\db{L}f$ of $\bigO\db{L}e$ as an $L$-algebra if and only if $\pi(D)$ is a defect group for $\bigO\db{L}f$ as a $G$-algebra. This explains why we are allowed to treat blocks of $\bigO\db{L}e$ as $G$-algebras below. 
\end{remark}

\begin{theorem}[{Profinite Fong-Reynolds reduction, cf. \cite[Theorem~6.8.13]{LinckelmannBookVol2}}]\label{thm:fong-reynolds}
    Let $G$ be a profinite group and $N$ a closed normal subgroup of $G$.  Let $S=\bigO\db{N}c$ be a $G$-stable block with trivial defect group.
    Then there is an open normal subgroup $L$ of $G$ and a map $s : G/L\to S^{\times}$ such that ${}^g(-) = {}^{s(g)}(-)$ as automorphisms of $S$ for every $g\in G$, and such that $s(gL)s(hL) = s(ghL)$ when at least one of $g,h\in N$.
    The map $\alpha : G/NL\times G/NL\to \bigO^\times$ given by $\alpha(xNL,yNL) = s(xyL)(s(xL)s(yL))^{-1}$ is well-defined  and a $2$-cocycle with image contained in the $p'$-roots of unity of $\bigO$.  We obtain an isomorphism of $\bigO$-algebras
    $$\bigO\db{G}c\iso S\otimes_{\bigO} \bigO_\alpha\db{G/N}$$
    sending $xc$ to $s(xL)\otimes xN$ for $x\in G$.

    If $N$ is finite, then $L$ can be chosen to be any open normal subgroup of $G$ intersecting $N$ trivially.
\end{theorem}

\begin{proof}
    First observe that $N$ must have finite $p$-Sylow subgroup because the (trivial) defect group of $S$ is open in a $p$-Sylow subgroup that contains it.  By Theorem~\ref{theorem:reduction_to_virtually_pro_p}, there is a closed normal pro-$p'$ subgroup $M$ of $N$ such that $N/M$ is virtually pro-$p$, hence finite, with $\bigO\db{N}c \cong \bigO[N/M]\varphi_M(c)$ via the canonical map $\varphi_M$, and such that $\bigO[N/M]\varphi_M(c)$ has trivial defect group.
    
    Hence, there exists an open normal subgroup $L$ of $G$ such that $L\cap N \subseteq M$. Note that in the special case where $N$ is finite, we can choose $M=1$ and the condition on $L$ becomes $L\cap N = 1$, which is the final claim of the theorem.
    Using Theorem~\ref{theorem:reduction_to_virtually_pro_p} again we see that $\bigO\db{N}c \cong \bigO[N/L\cap N]\varphi_{L\cap N}(c)$ via the canonical map.  In particular, $L$ acts trivially on $S = \bigO\db{N}c$ by conjugation and $xc=c$ for all $x\in L\cap N$.  

    Fix a Sylow $p$-subgroup $P$ of $G$ and set $Q=P\cap N$. Then $Q$ is a Sylow $p$-subgroup of $N$ and $G=N_G(Q)N$ by the Frattini argument. Let $K$ denote the fraction field of $\bigO$, set $\hat S=K\otimes_\bigO S$ and
    $$f_Q= \frac{1}{[Q:L\cap Q]}\sum_{y(L\cap Q)\in Q/(L\cap Q)} yc\in \hat S.$$ 
    If $V$ denotes an irreducible $S$-lattice, then $V$ is projective and the $p$-part of $\rank_{\bigO}(V)$ is equal to $[Q:L\cap Q]$ (this is all because $S$ is a block of trivial defect of $\bigO[N/L\cap N]$ { -- see for instance \cite[Theorem~6.6.2]{LinckelmannBookVol2}}). Then $V$ is free as an {$\bigO[Q/(Q\cap L)]$}-module so the rank of $f_QV = V^{Q/(L\cap Q)}$ is $\rank_{\bigO}(V)/[Q:L\cap Q]$, which is coprime to $p$. It follows that $f_Q\hat Sf_Q\cong M_n(K)$ for some $n$ coprime to $p$. In particular, since $\bigO^\times$ is closed under taking $n$-th roots, any element of $f_Q\hat Sf_Q$ with determinant in $\bigO^\times$ can be multiplied by an element of $\bigO^\times$ to obtain an element of determinant~$1$.

    For $x\in N_G(Q)$, consider $xL\in G/L$, and fix an element $t(xL)\in S^{\times}$ such that ${}^x(-) = {}^{t(xL)}(-)$.  Since $xL$ commutes with $f_Q$, it follows that $t(xL)$ is an element of $f_Q\hat Sf_Q\times (1-f_Q)\hat S(1-f_Q)$. By the last sentence of the preceding paragraph, we can assume that $f_Q t(xL)$ has determinant $1$ in $f_Q\hat Sf_Q\cong M_n(K)$. 
    
    Now $NL/L\unlhd G/L$ and we can argue as in the finite result \cite[Theorem~6.8.13]{LinckelmannBookVol2}: we choose a set $R$ of left coset representatives of $NL/L$ in $G/L$, and assume that $R \subset N_G(Q)L/L$ (where $Q$ is as above and we again rely on $G=N_G(Q)N$). For $xL\in R$ define $s(xL)=t(xL)$. We then extend to a map $s : G/L\to S^{\times}$ via $s(xnL) = s(xL)n\cdot c$ whenever $xL\in R$ and $n\in N$.  
    By the finite result the map $\alpha_L : G/L\times G/L\to \bigO^{\times}$ given by 
    \begin{equation}
        \alpha_L(xL,yL) = s(xyL)(s(xL)s(yL))^{-1}\label{eqn:cocyle}
    \end{equation} 
    is well defined and constant of cosets of $N$, that is, it induces a well-defined $2$-cocycle $\alpha : G/LN\times G/LN\to \bigO^{\times}$. We can choose coset representatives for $G/LN$ contained in $N_G(Q)$, which means that the image of $\alpha$ is contained in elements of the form $t(xyL)(t(xL)t(yL))^{-1}$. By our choice of the $t(xL)$, it follows that $f_Q\alpha(xLN, yLN)$ has determinant $1$ in $f_Q\hat Sf_Q\cong M_n(K)$. Given that it is also a scalar matrix, it follows that the image of $\alpha$ is contained in the $n$-th roots of unity in $\bigO^\times$.
    
    By the finite result, the map $gL\varphi_L(c)\mapsto s(gL)\otimes gLN$ induces an isomorphism of algebras $$\bigO[G/L]\varphi_L(c)\iso S\otimes_{\bigO} \bigO_\alpha[G/LN].$$
    To be precise, multiplicativity follows from equation~\eqref{eqn:cocyle} and surjectivity follows from $s(nL)=nc$.
    Consider now the cofinal system of open normal subgroups $M$ of $G$ contained in $L$.  For each such $M$ we can construct an isomorphism $\bigO[G/M]\varphi_M(c)\iso S\otimes_{\bigO} \bigO_\alpha[G/MN]$ by $gM\varphi_M(c)\mapsto s(gL)\otimes gMN$. Again, multiplicativity follows from equation~\eqref{eqn:cocyle} and surjectivity follows from $s(nL)=nc$.  
    This family of isomorphisms induces an isomorphism between inverse systems, and therefore between $\varprojlim_M\bigO[G/M]\varphi_M(c)=\bigO\db{G}c$ and $\varprojlim_M S\otimes_{\bigO} \bigO_\alpha[G/MN] = S\otimes_{\bigO} \bigO_\alpha\db{G/N}$ given by $gc\mapsto s(gL)\otimes gN$.
\end{proof}

\begin{corol}\label{corol:fong-reynolds}
    Let $G$ be a profinite group and let $\bigO \db{G}b$ be a block with defect group $D$. Let $N$ be a closed normal subgroup and let $S = \bigO \db{N}c$ be a $G$-stable block with trivial defect group such that $bc\neq 0$. Then there exists a central extension of profinite groups
    $$
        1 \longrightarrow Z \longrightarrow L \longrightarrow G/N \longrightarrow 1
    $$
    where $Z$ is a finite abelian $p'$-group, and a block $\bigO \db{L} e$ which is Morita equivalent to $\bigO \db{G}b$. The defect group of $\bigO \db{L} e$ is isomorphic to $D$.
\end{corol}

\begin{proof}
    By Proposition~\ref{prop:defectgroupintersection}, $D\cap N$ is conjugate within $G$ to a defect group of $\bigO \db{N}c$. That is, $D\cap N =1$. 
    By Theorem~\ref{thm:fong-reynolds}, we get an isomorphism of $\bigO$-algebras
    $$\gamma:\ \bigO\db{G}c\iso S\otimes_{\bigO} \bigO_\alpha\db{G/N}:\ xc \mapsto s(xL)\otimes xN$$
    for a suitable $2$-cocycle $\alpha$. In particular, there is a block idempotent $d$ of $\bigO_\alpha\db{G/N}$ such that $\gamma(b)=1_S\otimes d$, and $\bigO\db{G}b$ is Morita equivalent to $\bigO_\alpha\db{G/N}d$ (in fact, it is a matrix algebra over it).     
    Since $\gamma$ is an isomorphism of $G$-algebras (where $G$ acts on $S$ via $s(xL)$), it is clear that $D$ is a defect group of $S\otimes_{\bigO} \bigO_\alpha\db{G/N}d$. 
    
    We will show that $D$ is also a defect group of $\bigO_{\alpha}\db{G/N}d$. {Let $H\leq_C G$ and $M \unlhd_O G$.} Note that $N$ acts trivially on $\bigO_\alpha\db{G/N}$ and therefore $$(S\otimes_\bigO\bigO_\alpha\db{G/N})^N=S^N\otimes_\bigO\bigO_\alpha\db{G/N}=Z(S)\otimes_\bigO \bigO_\alpha\db{G/N}.$$ 
    Of course, $Z(S)=\bigO$, so any element in $Z(S)\otimes_\bigO \bigO_\alpha\db{G/N}$ is of the form $1\otimes y$. It follows that given $x\in (S\otimes_\bigO\bigO_\alpha\db{G/N})^{HM}$, we get $\Tr_{HM}^{HMN}(x)=1\otimes y$ for some $y\in \bigO_\alpha\db{G/N}^{HMN}$ and therefore
    $$
        \Tr_{HM}^G (x) = \Tr_{HMN}^G(1\otimes y) = 1 \otimes \Tr_{HMN/N}^{G/N}(y).
    $$
    Thus $\oTr_H^G((S\otimes \bigO_\alpha\db{G/N})^H)\subseteq Z(S)\otimes \oTr_{HN/N}^{G/N}(\bigO_\alpha\db{G/N}^{HN/N})$. Taking $H = D$, this shows that $DN/N$ contains a defect group of $\bigO_\alpha\db{G/N}d$ as a $G/N$-algebra. Conversely, if $E/N\subseteq DN/N$ is a defect group of $\bigO_\alpha\db{G/N}d$ as a $G/N$-algebra (where $N\leq E \leq G$), then $d\in \oTr_{E/N}^{G/N}(\bigO_\alpha\db{G/N}^{E/N})=\oTr_{E}^{G}(\bigO_\alpha\db{G/N}^{E})$, and therefore $1\otimes d \in \oTr_{E}^{G}((S\otimes_\bigO \bigO_\alpha\db{G/N})^{E})$. So $E$ contains a defect group of $\bigO\db{G}b$ as a $G$-algebra. We conclude that ${}^gD\subseteq E \subseteq DN$ for some $g\in G$.  
    So $E=DN$ because $DN$ cannot be conjugated to a proper subgroup of itself,  and therefore $E/N\cong DN/N\cong D$, where the last isomorphism uses $D\cap N=1$.

   The 2-cocycle $\alpha$ satisfies the hypotheses of Proposition~\ref{prop:centralextension} by Theorem \ref{thm:fong-reynolds}, so by Proposition~\ref{prop:centralextension} there is a central extension of profinite groups $L$ as in the statement and a block $\bigO \db{L}f$ such that $\bigO \db{L}f\cong \bigO_\alpha\db{G/N}d$ as $G/N$-algebras. In particular,  $\bigO \db{L}f$ has defect group isomorphic to $DN/N\cong D$ (see also Remark~\ref{remark:twistedgroupalgebra}) and it is Morita equivalent to $\bigO \db{G}b$. 
\end{proof}

\begin{remark}\label{rem:GbyOpOpprimeDecreases}
    Under the hypotheses of Corollary \ref{corol:fong-reynolds}, denote by $\pi : L \to G/N$ the given homomorphism.  Then 
    $$\pi^{-1}(O_p(G)N/N) \subseteq O_p(L)Z \subseteq O_p(L)O_{p'}(Z(L))$$
    and hence, if $G$ is virtually pro-$p$ (so the numbers are finite), $O_{p'}(Z(G)) \subsetneqq N$ and $N\cap O_p(G) = 1$, we have
    $$|L/O_p(L)O_{p'}(Z(L))| \leqslant |G/O_p(G)N| < |G/O_p(G)O_{p'}(Z(G))|.$$
\end{remark}

The results proved in this section so far hold for arbitrary profinite groups, but 
Theorem~\ref{thm:normal-defect} below will only be for virtually pro-$p$ groups. One of the technical difficulties in the general case, in particular in the proof of Proposition~\ref{prop:fong1}, was that in general a quotient map $\varphi_M:\ \bigO\db{G}\longrightarrow \bigO[G/M]$ (some $M\unlhd_O G$) does not send block idempotents to block idempotents. Proposition~\ref{prop:blockbijection} shows that this difficulty can be avoided when $G$ is virtually pro-$p$ (cf.\ Remark \ref{remark:vpropeasier}).

\begin{prop}\label{prop:blockbijection}
    Let $G$ be a virtually pro-$p$ group. Then there exists an open normal pro-$p$ subgroup $N_0$ of $G$ such that for any $N\unlhd_O G$ with $N\subseteq N_0$ the canonical map
    $$
        \varphi_N:\ Z(\bigO\db{G})  \longrightarrow Z(\bigO[G/N])
    $$
    induces a bijection of block idempotents.
\end{prop}

\begin{proof}
	The simple $\bigO\db{G}$-modules are in bijection with the simple $\bigO[G/N]$-modules by restriction along the quotient map, for any open normal pro-$p$ subgroup $N$ of $G$, so in particular $\bigO\db{G}$ has finitely many simple modules.

    Let $S$ and $T$ be simple $\bigO\db{G}$-modules lying in the same block of $\bigO\db{G}$. By \cite[Lemma~4.5]{JohnRicardoBrauer1}, there exists an open normal subgroup $N_{S,T}\unlhd G$, {which we may choose to be pro-$p$}, such that $S$ and $T$ lie in the same block of $\bigO[G/N_{S,T}]$. They will then also lie in the same block of $\bigO[G/N]$ for any $N\unlhd_O G$ with $N\subseteq N_{S,T}$. Conversely, if $S$ and $T$ lie in the same block of a continuous finite quotient of $G$, then they lie in the same block of $\bigO\db{G}$. 
    
    We define
    $$
        N_0= \bigcap_{S,T} N_{S,T},
    $$
    where $S$ and $T$ range over all pairs of simple $\bigO\db{G}$-modules lying in the same block of $\bigO\db{G}$. This means that for any $N\subseteq N_0$ with $N\unlhd_O G$, two simple $\bigO[G/N]$-modules lie in the same block if and only if they lie in the same block of $\bigO\db{G}$. Moreover, any simple $\bigO\db{G}$-module arises from a simple $\bigO[G/N]$-module by restriction. Hence $\bigO\db{G}$ and $\bigO[G/N]$ have the same number of blocks and the map $\varphi_N$ does not map any block idempotent to zero, so that $\varphi_N$ induces a bijection of block idempotents.
\end{proof}

\begin{theorem}\label{thm:normal-defect}
    Let $G$ be a virtually pro-$p$ group and let $\bigO \db{G}b$ be a block with normal defect group $D$. Assume that $C_G(D)=Z(D)A$ for a (necessarily finite) $p'$-group $A\unlhd G$. Then there is a finite $p'$-group $E$ acting faithfully on $D$ and a $2$-cocycle $\alpha\in Z^2(E,\bigO^\times)$ with values in the $p'$-roots of unity such that $\bigO\db{G}b$ is Morita equivalent to $\bigO_\alpha\db{D\rtimes E}$.
\end{theorem}
\begin{proof}
    Pick a block idempotent $e\in Z(\bigO\db{DC_G(D)})$ such that $be\neq 0$. 
    By Proposition~\ref{prop:fong1} the block $\bigO\db{G}b$ is Morita equivalent to a block of the stabiliser of $e$ in $G$. We can therefore assume without loss of generality that $e$ is $G$-stable.

    We will first show that the inertial quotient of $\bigO \db{G}b$, that is, $G/DC_G(D)$, is equal to the inertial quotient of the block of a finite quotient. In particular, it will be a finite $p'$-group. 
    If $N$ is an open normal subgroup of $G$ contained in $D$, we
    write $C_G(D/N)$ for the preimage in $G$ of $C_{G/N}(D/N)$. Since $C_G(D)$ is closed we get $$C_G(D/N)D=C_G(D)D$$ for some sufficiently small $N$.
    By Proposition~\ref{prop:blockbijection} we can, in addition to the above, choose $N$ small enough that $\varphi_N(b)\in Z(\bigO[G/N])$ and $\varphi_N(e)\in Z(\bigO[DC_G(D)/N])$ are both block idempotents, which we denote $b_N$ and $e_N$. Choosing $N$ smaller if needed we can furthermore assume that $D/N$ is a defect group of $\bigO[G/N]b_N$.

    Note that $be\neq 0$ implies $b_N e_N\neq 0$. Furthermore, $\Stab_G(e_N)=\Stab_G(e)=G$ is also clear by Proposition~\ref{prop:blockbijection}. We can now apply Brauer's First Main Theorem \cite[Theorem~6.7.6~(v)]{LinckelmannBookVol2} (combined with \cite[Corollary~6.2.7]{LinckelmannBookVol2}) to conclude that $E=G/DC_G(D)$ is a $p'$-group.

    Since $C_G(D)=Z(D)A$ for a finite $p'$-group $A\unlhd G$, we have that $G/A \cong D\rtimes E$ by the profinite Schur-Zassenhaus theorem. We have $DC_G(D) = DZ(D)A = D\times A$, and hence $e = 1_{\bigO\db{D}}\otimes c$ for a block idempotent $c$ of $\bigO A$, which is thus $G$-stable, of trivial defect and such that $bc\neq 0$. 
    
    We are in a position to apply Theorem~\ref{thm:fong-reynolds} with the $N$ from the statement being our $A$.  Since $A$ is finite and $D\cap A = 1$, we can take $L=D$.  Thus, $\bigO \db{G}b$ is Morita equivalent to a block of $\bigO_\alpha \db{G/A}$ for a $2$-cocycle $\alpha \in Z^2(G/DA, \bigO^\times) = Z^2(E,\bigO^\times)$ with values in the $p'$-roots of unity. It follows from \cite[Theorem~6.14.1]{LinckelmannBookVol2} that $\bigO_{\alpha}[(D\rtimes E)/N]$ is indecomposable as an $\bigO$-algebra, where $N$ is as above ($E$ continues to act faithfully on $D/N$ because $C_G(D/N)D = C_G(D)D$). But then $\bigO_{\alpha}\db{D\rtimes E}$ is indecomposable too, which means that $\bigO \db{G}b$ is Morita equivalent to $\bigO_\alpha\db{D\rtimes E}$ itself, rather than just a block thereof.
\end{proof}

The above theorem is a version for virtually pro-$p$ groups of part of \cite[Theorem~6.14.1]{LinckelmannBookVol2}.  It would be interesting to prove the full analogue of \cite[Theorem~6.14.1]{LinckelmannBookVol2} -- this would require much of the theory of interior algebras and source algebras to be generalised to the profinite setting (which could be a lucrative pursuit in other contexts as well). It would also be interesting to see if there is a way to generalise Theorem \ref{thm:normal-defect} from virtually pro-$p$ groups to arbitrary profinite groups.

Given a finite group $E$, in the following we identify $H^2(E, k^\times)$ with the subgroup of $H^2(E, \bigO^\times)$ consisting of classes of $2$-cocycles taking values in the $p'$-roots of unity. 

\begin{theorem}\label{thm:reduction}
    Let $G$ be a profinite group and let $\bigO \db{G}b$ be a block with defect group isomorphic to $\Z_p^n$ for some $n\in \N$. Then there is a finite $p'$-group $E$ acting faithfully on $\Z_p^n$ and an $\alpha\in H^2(E,k^\times)$ such that $\bigO\db{G}b$ is Morita equivalent to $\bigO_\alpha\db{\Z_p^n\rtimes E}$.
\end{theorem}
\begin{proof}
    By Theorem~\ref{theorem:reduction_to_virtually_pro_p} we can assume that $G$ is virtually pro-$p$. We further assume that among the virtually pro-$p$ groups having a block Morita equivalent to $\bigO\db{G}b$ with defect group $\Z_p^n$, $G$ is chosen with $|G/O_p(G)O_{p'}(Z(G))|$ as small as possible.
    
    We can now apply Lemma~\ref{lemma:group-theory-caseAB} and we claim that Case (A) cannot occur: if it did, then let $N$ be the given finite normal subgroup.  We have a block idempotent $c\in \bigO[N]$ with $bc\neq 0$, which has trivial defect group by Proposition~\ref{prop:defectgroupintersection}.  We claim $c$ is $G$-stable: By Proposition~\ref{prop:fong1}, the block $\bigO\db{G}b$ is Morita equivalent to a block $\bigO\db{H}d$ with defect group $\Z_p^n$, where $H=\Stab_G(c)$ and $bd\neq 0$. Note that $H$ contains a defect group of $\bigO\db{G}b$ (which contains $O_p(G)$) and $N$ (which contains $O_{p'}(Z(G))$), so $H\supseteq O_p(G)O_{p'}(Z(G))$. This implies $|H/O_p(H)O_{p'}(Z(H))| \leq |H/O_p(G)O_{p'}(Z(G))| \leq |G/O_p(G)O_{p'}(Z(G))|$, and the second inequality is strict if $H\neq G$. So by our minimality assumption, $H=G$, as claimed.  
     Now Corollary~\ref{corol:fong-reynolds} yields a virtually pro-$p$ group $L$ with (by Remark \ref{rem:GbyOpOpprimeDecreases}) $|L/O_p(L)O_{p'}(Z(L))| < |G/O_p(G)O_{p'}(Z(G))|$, and a block $\bigO\db{L} e$ Morita equivalent to $\bigO\db{G}b$ with the same defect group, contradicting our minimality assumption.
    
    We are thus in Case (B), and Theorem~\ref{thm:normal-defect} implies the assertion.
\end{proof}

\begin{corol}\label{corol:DonovanforZpn}
    There are only finitely many Morita equivalence classes of blocks with defect group isomorphic to $\Z_p^n$, for any given $n$.
\end{corol}
\begin{proof}
    There are only finitely many possibilities for $E$ in Theorem~\ref{thm:reduction}. This is due to the fact that any such $E$ is isomorphic to a subgroup of $\GL_n(\mathbb F_p)$, of which there are only finitely many, and an action of $E$ on $\Z_p^n$ corresponds to a $\Z_pE$-lattice of rank $n$, of which there are again only finitely many up to isomorphism. For any such $E$ the group $H^2(E,k^\times)$ is finite, so the result follows.
\end{proof}

\section*{Acknowledgements}

We would like to thank Pavel Zalesski and Henrique Souza for helpful discussions. The first author was supported by EPSRC grant UKRI2780 and FAPEMIG Universal Grant APQ-03491-25. The second author was supported by CNPq Produtividade 1D Grant 303667/2022-2, CNPq Universal Grant 317682/2025-3 and FAPEMIG Universal Grant APQ-03491-25.

\bibliographystyle{abbrv}
\bibliography{refpap}

@article {JohnRicardoCyclic,
    AUTHOR = {Franquiz Flores, R. J. and MacQuarrie, J. W.},
     TITLE = {Blocks of profinite groups with cyclic defect group},
   JOURNAL = {Bull. Lond. Math. Soc.},
  FJOURNAL = {Bulletin of the London Mathematical Society},
    VOLUME = {54},
      YEAR = {2022},
    NUMBER = {5},
     PAGES = {1595--1608},
      ISSN = {0024-6093},
   MRCLASS = {20C20 (20E18)},
  MRNUMBER = {4499593},
}

@article {JohnRicardoBrauer1,
    AUTHOR = {Franquiz Flores, Ricardo J. and MacQuarrie, John W.},
     TITLE = {Block theory and {B}rauer's first main theorem for profinite
              groups},
   JOURNAL = {Adv. Math.},
  FJOURNAL = {Advances in Mathematics},
    VOLUME = {397},
      YEAR = {2022},
     PAGES = {Paper No. 108121, 31},
      ISSN = {0001-8708},
   MRCLASS = {20C20 (20E18)},
  MRNUMBER = {4377937},
       DOI = {10.1016/j.aim.2021.108121},
       URL = {https://doi.org/10.1016/j.aim.2021.108121},
}

@article {JohnPavelPeterInfGen,
    AUTHOR = {MacQuarrie, John W. and Symonds, Peter and Zalesskii, Pavel
              A.},
     TITLE = {Infinitely generated pseudocompact modules for finite groups
              and {W}eiss' theorem},
   JOURNAL = {Adv. Math.},
  FJOURNAL = {Advances in Mathematics},
    VOLUME = {361},
      YEAR = {2020},
     PAGES = {106925, 35},
      ISSN = {0001-8708},
   MRCLASS = {20C12 (16G30)},
  MRNUMBER = {4038145},
MRREVIEWER = {Nadia P. Mazza},
       DOI = {10.1016/j.aim.2019.106925},
       URL = {https://doi.org/10.1016/j.aim.2019.106925},
}

@book {LinckelmannBookVol2,
    AUTHOR = {Linckelmann, Markus},
     TITLE = {The block theory of finite group algebras. {V}ol. {II}},
    SERIES = {London Mathematical Society Student Texts},
    VOLUME = {92},
 PUBLISHER = {Cambridge University Press, Cambridge},
      YEAR = {2018},
     PAGES = {xi+510},
      ISBN = {978-1-108-42590-2; 978-1-108-44180-3; 978-1-108-44190-2},
   MRCLASS = {20C20 (16-02 16D90 20-02)},
  MRNUMBER = {3821517},
MRREVIEWER = {Burkhard K\"{u}lshammer},
}

@book {LinckelmannBookVol1,
    AUTHOR = {Linckelmann, Markus},
     TITLE = {The block theory of finite group algebras. {V}ol. {I}},
    SERIES = {London Mathematical Society Student Texts},
    VOLUME = {91},
 PUBLISHER = {Cambridge University Press, Cambridge},
      YEAR = {2018},
     PAGES = {x+515},
      ISBN = {978-1-108-42591-9; 978-1-108-44183-4; 978-1-108-44190-2},
   MRCLASS = {20C20 (16-02 16S34 16T05 20-02)},
  MRNUMBER = {3821516},
MRREVIEWER = {Burkhard K\"{u}lshammer},
}

@article{IusenkoMacQuarrieSemisimple,
author = {Iusenko, Kostiantyn and MacQuarrie, John William},
title = {Semisimplicity and separability for pseudocompact algebras},
journal = {J.\ Algebra Appl.},
volume = {24},
number = {03},
pages = {2550078},
year = {2025},
doi = {10.1142/S0219498825500781},
}

@article {GabrielCategoriesAbeliennes,
    AUTHOR = {Gabriel, Pierre},
     TITLE = {Des cat\'{e}gories ab\'{e}liennes},
   JOURNAL = {Bull. Soc. Math. France},
  FJOURNAL = {Bulletin de la Soci\'{e}t\'{e} Math\'{e}matique de France},
    VOLUME = {90},
      YEAR = {1962},
     PAGES = {323--448},
      ISSN = {0037-9484},
   MRCLASS = {18.20},
  MRNUMBER = {232821},
MRREVIEWER = {T.-Y. Lam},
       URL = {http://www.numdam.org/item?id=BSMF_1962__90__323_0},
}

@book {RibesZalesskiiProfiniteGroups,
    AUTHOR = {Ribes, Luis and Zalesskii, Pavel},
     TITLE = {Profinite groups},
    SERIES = {Ergebnisse der Mathematik und ihrer Grenzgebiete. 3. Folge. A
              Series of Modern Surveys in Mathematics [Results in
              Mathematics and Related Areas. 3rd Series. A Series of Modern
              Surveys in Mathematics]},
    VOLUME = {40},
 PUBLISHER = {Springer-Verlag, Berlin},
      YEAR = {2000},
     PAGES = {xiv+435},
      ISBN = {3-540-66986-8},
   MRCLASS = {20E18},
  MRNUMBER = {1775104},
MRREVIEWER = {Marcus du Sautoy},
       DOI = {10.1007/978-3-662-04097-3},
       URL = {https://doi.org/10.1007/978-3-662-04097-3},
}

@article {JohnPeterBrauerTheory,
    AUTHOR = {MacQuarrie, John W. and Symonds, Peter},
     TITLE = {Brauer theory for profinite groups},
   JOURNAL = {J. Algebra},
  FJOURNAL = {Journal of Algebra},
    VOLUME = {398},
      YEAR = {2014},
     PAGES = {496--508},
      ISSN = {0021-8693},
   MRCLASS = {20C20 (20E18)},
  MRNUMBER = {3123781},
MRREVIEWER = {Nadia P. Mazza},
       DOI = {10.1016/j.jalgebra.2013.09.004},
       URL = {https://doi.org/10.1016/j.jalgebra.2013.09.004},
}

@article{IusenkoMacQuarrieExtensions,
  author    = {Iusenko, Kostiantyn and MacQuarrie, John William},
  title     = {Homological properties of extensions of algebras},
  JOURNAL = {Math. Z.},
  volume    = {309},
  number    = {1},
  pages     = {55},
  year      = {2025},
  publisher = {Springer},
  doi       = {10.1007/s00209-025-03684-z},
  url       = {https://link.springer.com/article/10.1007/s00209-025-03684-z}
}

@misc{ColinReidThesis,
      title={Finiteness properties of profinite groups}, 
      author={Colin Reid},
      year={2010},
      eprint={1002.2935},
      archivePrefix={arXiv},
      primaryClass={math.GR},
      url={https://arxiv.org/abs/1002.2935}, 
      note={{PhD thesis}}
}

@article {Brumer,
    AUTHOR = {Brumer, Armand},
     TITLE = {Pseudocompact algebras, profinite groups and class formations},
   JOURNAL = {J. Algebra},
  FJOURNAL = {Journal of Algebra},
    VOLUME = {4},
      YEAR = {1966},
     PAGES = {442--470},
      ISSN = {0021-8693},
   MRCLASS = {18.20},
  MRNUMBER = {202790},
MRREVIEWER = {H.\ Bass},
       DOI = {10.1016/0021-8693(66)90034-2},
       URL = {https://doi-org.manchester.idm.oclc.org/10.1016/0021-8693(66)90034-2},
}

@incollection {KrauseVossieck,
    AUTHOR = {Krause, Henning and Vossieck, Dieter},
     TITLE = {Length categories of infinite height},
 BOOKTITLE = {Geometric and topological aspects of the representation theory
              of finite groups},
    SERIES = {Springer Proc. Math. Stat.},
    VOLUME = {242},
     PAGES = {213--234},
 PUBLISHER = {Springer, Cham},
      YEAR = {2018},
      ISBN = {978-3-319-94033-5; 978-3-319-94032-8},
   MRCLASS = {18E10},
  MRNUMBER = {3901161},
MRREVIEWER = {Gustavo\ Jasso},
       DOI = {10.1007/978-3-319-94033-5\_8},
       URL = {https://doi-org.manchester.idm.oclc.org/10.1007/978-3-319-94033-5_8},
}

@article {Bongartz,
    AUTHOR = {Bongartz, Klaus},
     TITLE = {Algebras and quadratic forms},
   JOURNAL = {J. London Math. Soc. (2)},
  FJOURNAL = {Journal of the London Mathematical Society. Second Series},
    VOLUME = {28},
      YEAR = {1983},
    NUMBER = {3},
     PAGES = {461--469},
      ISSN = {0024-6107,1469-7750},
   MRCLASS = {16A64 (16A46)},
  MRNUMBER = {724715},
MRREVIEWER = {Sheila\ Brenner},
       DOI = {10.1112/jlms/s2-28.3.461},
       URL = {https://doi-org.manchester.idm.oclc.org/10.1112/jlms/s2-28.3.461},
}

@book {SerreGaloisCohomology,
    AUTHOR = {Serre, Jean-Pierre},
     TITLE = {Galois cohomology},
    SERIES = {Springer Monographs in Mathematics},
   EDITION = {English},
      NOTE = {Translated from the French by Patrick Ion and revised by the
              author},
 PUBLISHER = {Springer-Verlag, Berlin},
      YEAR = {2002},
     PAGES = {x+210},
      ISBN = {3-540-42192-0},
   MRCLASS = {12G05 (11R34)},
  MRNUMBER = {1867431},
}

@book {WilsonProfiniteGroups,
    AUTHOR = {Wilson, John S.},
     TITLE = {Profinite groups},
    SERIES = {London Mathematical Society Monographs. New Series},
    VOLUME = {19},
 PUBLISHER = {The Clarendon Press, Oxford University Press, New York},
      YEAR = {1998},
     PAGES = {xii+284},
      ISBN = {0-19-850082-3},
   MRCLASS = {20E18},
  MRNUMBER = {1691054},
MRREVIEWER = {Alexander\ Lubotzky},
}

@book{isaacs2008finite,
  title={Finite Group Theory},
  author={Isaacs, I. Martin},
  series={Graduate Studies in Mathematics},
  volume={92},
  year={2008},
  publisher={American Mathematical Society},
  address={Providence, RI},
  isbn={978-0-8218-4344-4}
}

@article {CapraceReidEtAl,
    AUTHOR = {Caprace, Pierre-Emmanuel and Reid, Colin and Wesolek, Phillip},
     TITLE = {Approximating simple locally compact groups by their dense
              locally compact subgroups},
   JOURNAL = {Int. Math. Res. Not. IMRN},
  FJOURNAL = {International Mathematics Research Notices. IMRN},
      YEAR = {2021},
    NUMBER = {7},
     PAGES = {5037--5110},
      ISSN = {1073-7928,1687-0247},
   MRCLASS = {22D05},
  MRNUMBER = {4241124},
MRREVIEWER = {Bruno\ Duchesne},
       DOI = {10.1093/imrn/rny298},
       URL = {https://doi-org.manchester.idm.oclc.org/10.1093/imrn/rny298},
}

@misc{ProfusionPaper,
      title={Block pro-fusion systems for profinite groups and blocks with infinite dihedral defect groups}, 
      author={Florian Eisele and Ricardo J. Franquiz Flores and John W. MacQuarrie},
      year={2025},
      eprint={2504.09286},
      archivePrefix={arXiv},
      primaryClass={math.RT},
      url={https://arxiv.org/abs/2504.09286}, 
}

@article {ChinPathCoalgebras,
    AUTHOR = {Chin, William},
     TITLE = {Hereditary and path coalgebras},
   JOURNAL = {Comm. Algebra},
  FJOURNAL = {Communications in Algebra},
    VOLUME = {30},
      YEAR = {2002},
    NUMBER = {4},
     PAGES = {1829--1831},
      ISSN = {0092-7872,1532-4125},
   MRCLASS = {16W30},
  MRNUMBER = {1894047},
       DOI = {10.1081/AGB-120013219},
       URL = {https://doi-org.manchester.idm.oclc.org/10.1081/AGB-120013219},
}

@article{JohnGreenCorr,
	author = {MacQuarrie, John W.},
	title = {Green Correspondence for Virtually Pro-$p$ Groups},
	JOURNAL = {J. Algebra},
	volume = {323},
	year = {2010},
	pages = {2203-2208},
	doi = {10.1016/j.jalgebra.2010.02.011}
}

@article{JohnModReps,
	author={MacQuarrie, John W.},
	title={Modular representations of profinite groups},
	journal={J.\ Pure Appl.\ Algebra},
	volume={215},
	year={2011},
	number={5},
	pages={753--763},
	DOI={10.1016/j.jpaa.2010.06.022}
}

@article{EEL,
    AUTHOR = {Eaton, C.W. and Eisele, F. and Livesey, M.},
     TITLE = {Donovan's conjecture, blocks with abelian defect groups and
              discrete valuation rings},
   JOURNAL = {Math. Z.},
  FJOURNAL = {Mathematische Zeitschrift},
    VOLUME = {295},
      YEAR = {2020},
    NUMBER = {1-2},
     PAGES = {249--264},
      ISSN = {0025-5874},
   MRCLASS = {20C20 (13A18)},
  MRNUMBER = {4100018},
MRREVIEWER = {Gabriel Navarro},
       DOI = {10.1007/s00209-019-02354-1},
       URL = {https://0-doi-org.wam.city.ac.uk/10.1007/s00209-019-02354-1},
}

@article{LinckelmannV4,
title = {The Source Algebras of Blocks with a {K}lein Four Defect Group},
JOURNAL = {J. Algebra},
volume = {167},
number = {3},
pages = {821-854},
year = {1994},
issn = {0021-8693},
doi = {https://doi.org/10.1006/jabr.1994.1214},
url = {https://www.sciencedirect.com/science/article/pii/S0021869384712142},
author = {M. Linckelmann}
}

@article {ErdmannV4,
    AUTHOR = {Erdmann, Karin},
     TITLE = {Blocks whose defect groups are {K}lein four groups: a
              correction},
   JOURNAL = {J. Algebra},
  FJOURNAL = {Journal of Algebra},
    VOLUME = {76},
      YEAR = {1982},
    NUMBER = {2},
     PAGES = {505--518},
      ISSN = {0021-8693},
   MRCLASS = {20C20},
  MRNUMBER = {661869},
MRREVIEWER = {L.\ Dornhoff},
       DOI = {10.1016/0021-8693(82)90228-9},
       URL = {https://doi-org.manchester.idm.oclc.org/10.1016/0021-8693(82)90228-9},
}

@article {KuelshammerpSolvable,
    AUTHOR = {K\"ulshammer, Burkhard},
     TITLE = {On {$p$}-blocks of {$p$}-solvable groups},
   JOURNAL = {Comm. Algebra},
  FJOURNAL = {Communications in Algebra},
    VOLUME = {9},
      YEAR = {1981},
    NUMBER = {17},
     PAGES = {1763--1785},
      ISSN = {0092-7872,1532-4125},
   MRCLASS = {20C20},
  MRNUMBER = {631888},
MRREVIEWER = {Karin\ Erdmann},
       DOI = {10.1080/00927878108822682},
       URL = {https://doi-org.manchester.idm.oclc.org/10.1080/00927878108822682},
}

@unpublished{JohnPeterExtendedBrauerUnfinished,
  author    = {MacQuarrie, John W. and Symonds, Peter},
  title     = {Blocks of profinite group algebras and the extended {B}rauer correspondence},
  note      = {Manuscript in preparation},
  year      = {2026}
}

\end{document}